\documentclass{paper}[12pt] 
\usepackage{tikz}

\usepackage{amssymb}   
\usepackage{amsthm}
\usepackage{amsmath}
\usepackage{amsfonts}
\usepackage{latexsym} 
\usepackage{eucal}
\usepackage{indentfirst} 
\usepackage{graphics} 
\usepackage{verbatim} 

\usepackage{setspace}
\newtheorem{theorem}{Theorem}[section]

\newtheorem{prop}[theorem]{Proposition}
\newtheorem{lemma}[theorem]{Lemma}

\newtheorem{corr}[theorem]{Corollary}
\newtheorem{definition}[theorem]{Definition}

\def\R{{\mathbb R}} 
\def\N{{\mathbb N}} 

\begin{document}
\overfullrule=0pt
\baselineskip=24pt
\font\tfont= cmbx10 scaled \magstep3
\font\sfont= cmbx10 scaled \magstep2
\font\afont= cmcsc10 scaled \magstep2
\title{\tfont A multi-dimensional continued fraction generalization\\
of Stern's  diatomic sequence}
\bigskip
\bigskip
\author{Thomas Garrity\\  
Department of Mathematics and Statistics\\ Williams College\\ Williamstown, MA  01267\\ 
email:tgarrity@williams.edu}

\date{}

\maketitle
\begin{abstract}

Continued fractions are linked to Stern's diatomic sequence $0,1,1,2,1,3,2,3,1,4,\ldots $  (given by the recursion relation $\alpha_{2n}=\alpha_n$ and $\alpha_{2n+1} = \alpha_n + \alpha_{n+1},$ where $\alpha_0=0$ and $\alpha_1=1$), which has long been known.  Using a particular multi-dimensional continued fraction algorithm (the Farey algorithm), we will generalize the diatomic sequence to a sequence of numbers  that quite naturally should be   called Stern's  triatomic sequence (or a two-dimensional Pascal's sequence with memory).  As continued fractions and the diatomic sequence can be thought of as coming from a  systematic subdivision of the unit interval, this new triatomic sequence will arise by a systematic subdivision of a triangle.  We will discuss some of the algebraic properties for the triatomic sequence.   


\end{abstract}

\section{The Classical Stern Diatomic Sequence}
\begin{singlespacing}

The goal of this paper is to take a particular generalization of continued fractions (the Farey algorithm) and find the corresponding generalization of Stern's diatomic sequence, which we will call  Stern's triatomic sequence.
   In this first  section, though,  we will 
 quickly review the basics of  Stern's  diatomic sequence. Hence experts should skip to the next section. 
In  section \ref{sec:Stern's Triatomic sequence} we give a  formal presentation  of Stern's triatomic sequence.  Section \ref{sec:Stern's Triatomic Sequence and Triangles} shows how Stern's triatomic sequence can be derived from a systematic subdivision of a triangle.  Section \ref{sec:The Farey Map}
 shows how Stern's triatomic sequence arises from the multi-dimensional continued fraction stemming from the Farey map, in analog to the link between Stern's diatomic sequence and classical continued fractions.   Sections \ref{sec:Combinatorial Properties of Tri-Atomic Sequences} and \ref{sec:Paths of a directed graph} discuss some of the properties for this sequence.  Many of these properties are analogs of corresponding properties for Stern's diatomic sequence.  We will conclude with open questions.

 A recent Monthly article by Northshield \cite{Northshield1} gives a good overview of Stern's diatomic sequence and, more importantly for this paper, its relation to continued fractions.  In much earlier work, 
Lehmer \cite{Lehmer1} lists a number of properties of this sequence.  In part, our eventual  goal is  showing how these properties usually generalize for our new Stern's triatomic sequences.

Consider an interval, with endpoints weighted by numbers  $v_1$ and $v_2$:

\begin{center}
\begin{tikzpicture}
\draw[|-|] (0,0)--(5,0);
\node[below] at (0,0){$v_1$};
\node[below] at (5,0){$v_2$};
   \end{tikzpicture}
   \end{center}

Then subdivide:

\begin{center}
\begin{tikzpicture}
\draw[|-|] (0,0)--(2.5,0);
\draw[|-|] (5,0)--(2.5,0);
\node[below] at (0,0){$v_1$};
\node[below] at (5,0){$v_2$};
\node[below] at (2.5,0){$v_1 + v_2$};
   \end{tikzpicture}
   \end{center}

We can continue, getting:

\begin{center}
\begin{tikzpicture}
\draw[|-|] (0,0)--(1.66,0);
\draw[|-|] (1.66,0)--(2.5,0);
\draw[|-|] (3.38,0)--(2.5,0);
\draw[|-|] (3.38,0)--(5,0);

\node[below] at (0,0){$v_1$};
\node[below] at (5,0){$v_2$};
\node[below] at (2.5,0){$v_1 + v_2$};
\node[above] at (1.66,0){$2v_1+v_2$};
\node[above] at (3.38,0){$v_1+2v_2$};

 \end{tikzpicture}
   \end{center}

The traditional Stern diatomic sequence is doing this for the case when $v_1=v_2=1.$  If instead we start with vectors 
$$v_1= \left( \begin{array}{c} 0 \\ 1 \end{array} \right) ,\; v_2=\left( \begin{array}{c} 1 \\ 1 \end{array} \right) $$
and do the standard identification of 
$$\left( \begin{array}{c} x \\ y \end{array} \right) \rightarrow \frac{x}{y},$$
we get the Farey decomposition of the unit interval, which in turn can be used to recover the continued fraction expansion of any real number on the unit interval.   The case of letting $v_1=v_2=1$ is equivalent  to keeping track of the denominators for the Farey decomposition.  (We mention this as motivation for why we later link our new Stern's triatomic sequence to a particular multi-dimensional continued fraction algorithm.)

Thus  Stern's diatomic sequence is given by the recursion formulas
$$\alpha_0= 0, \alpha_1=1$$
and
\begin{eqnarray*}
\alpha_{2n}    &=& \alpha_n \\
\alpha_{2n+1} &=& \alpha_n + \alpha_{n+1}
\end{eqnarray*}
The first few terms are 
$$0,1,1,2,1,3,2,3,1,4,3,5,2,5,4,\ldots.$$
It is standard to write this as

$$ \begin{array} {ccccccccc}
\alpha_1 &  &  &  &  &  &  &  & \alpha_2 \\
\alpha_2 &  &  &  & \alpha_3 &  &  &  & \alpha_4 \\
\alpha_4 &  & \alpha_5 &  & \alpha_6 &  & \alpha_7 &  & \alpha_8 \\
\alpha_8 & \alpha_9 & \alpha_{10} & \alpha_{11} & \alpha_{12} & \alpha_{13} & \alpha_{14} & \alpha_{15} & \alpha_{16} \\ 
\end{array}$$

For $\alpha_1=1$, we get 
$$ \begin{array} {ccccccccccccccccc}
1 &&  && & &  &&  &&  &&  &&  && 1 \\
1 &&  &&  &&  && 2 &&  &&  &&  && 1 \\
1 &&  && 3 &&  && 2 &&  && 3 &&  && 1 \\
1 && 4 && 3 && 5 && 2 && 5 && 3 && 4 && 1 \\ 
1 &5& 4 &7& 3 &8& 5 &7& 2 &7& 5 &8& 3 &7& 4 &5& 1 \\ 
\end{array}$$

Thus to create the next level, we simply interlace the original level with the sum of adjacent entries.  This is why this sequence is called by Knauf \cite{Knauf1} {\it Pascal with Memory}.  (It was Knauf's work on linking this sequence with statistical mechanics, which is also in \cite{Kelban-Ozluk1, Knauf5, Knauf2, Knauf3, Knauf4}, that led the author to look at these sequences.  Other work linking statistical mechanics with Stern diatomic sequences, though these works do not use the rhetoric of Stern  diatomic sequences, include   \cite{Contucci-Knauf1, Esposti-Isola-Knauf1, Fiala-Kleban1, Fiala-Kleban-Ozluk1, Garrity10, Guerra-Knauf1, Kallies-Ozluk-Peter-Syder1, Kelban-Ozluk1, Mayer2, MendesFrance-Tenenbaum1, MendesFrance-Tenenbaum2, Prellberg1, Prellberg-Fiala-Kleban1, Prellberg-Slawny1})

A matrix approach for  Stern's diatomic sequence $\alpha_1, \alpha_2, \ldots $.  
is as follows.  Set 

$$b_0=  \left(  \begin{array}{cc} 1 & 1 \\ 0 & 1 \end{array}   \right), b_1= \left(  \begin{array}{cc}  1 & 0 \\ 1 & 1 \end{array}   \right)$$
and 
$$M= \left(  \begin{array}{cc} 0 & 1 \\ 1 & 1 \end{array}    \right).$$
Let $N >2$ be a positive integer.  Then there is a tuple $(i_k, \ldots, i_1)_S$ with $k$ a positive integer and each $i_j\in \{0,1\}$ such that 
$$N=2^{k} + 1 + i_{k}2^{k-1}+\cdots + i_1.$$
(The subscript $S$  stands for  `Stern'.) For example $3 = 2^1 + 1 + 0\cdot 1 = (0)_S, 
4 = 2+ 1  + 1\cdot 1 = (1)_S, 
5 = 2^2 + 1+ + 0 \cdot 2^1 + 0\cdot 1 =  (0,0)_S $,  $6=2^2 +1+  0\cdot 2 + 1\cdot 1 = (0,1)_S$,  $7= 2^2 + 1+ 1\cdot 2 + 0 \cdot 1 = (1,0)_S$, $8=(1,1)_S$  and 
$16 = 2^3 + 1 + 1 \cdot 2^2 + 1\cdot 2^1 + 1\cdot 1= (1,1, 1)_S.$
Then we have $$\alpha_N = \left(  \begin{array}{cc} 0 & 1 \end{array} \right) Mb_{i_k}\cdots b_{i_1}   \left(  \begin{array}{c} 0 \\ 1 \end{array} \right) .$$

This is just systemizing the subdivision of the unit interval via Farey fractions.  As described by the author \cite{Garrity10}, these matrices can be used to find the continued fraction expansion for any real number, which is in part why one would expect  any natural generalization of Stern's diatomic sequence to be linked with a generalization of continued fractions.

Finally, Allouche and Shallit \cite{Allouche-Shallit92, Allouche-Shallit03} have shown that  Stern's diatomic sequence is a $2$-regular sequence.

\section{Stern's Triatomic sequence} \label{sec:Stern's Triatomic sequence}
We now start with explaining the new sequence of this paper.
Start with an initial triple (which we will also think of as a $1\times 3$ row matrix)
$$\triangle = (v_1, v_2, v_3) = (1,1,1).$$
We are using the notation $\triangle$ to suggest the vertices of a triangle. For any index
$$I=(i_1, i_2, \ldots, i_n),$$
with each $i_k$ being zero, one or two, we will associate a triple which we denote by $\triangle(I).$ 

Suppose we know the triple
$$\triangle(I)=(v_1(I),v_2(I),v_3(I)).$$
Consider the three matrices:
$$A_0=  \left(  \begin{array}{rrr}    1  &  0  &  1 \\ 0  & 1 & 1 \\ 0  & 0 & 1 \end{array} \right),   A_1=   \left(  \begin{array}{rrr}    0  &  0  &  1 \\ 1  & 0 & 1 \\ 0  & 1 & 1 \end{array} \right)   A_2=   \left(  \begin{array}{rrr}    0  &  1 &  1 \\ 0  & 0 & 1 \\ 1 & 0 & 1 \end{array} \right).$$
Then we set 
$$\begin{array}{ccccc}
\triangle (I,0)  &=&  \triangle (I) A_0  &=&(v_1(I),v_2(I),v_3(I)) A_0  \\
\triangle (I,1)  &=&  \triangle (I) A_1  &=&(v_1(I),v_2(I),v_3(I)) A_1 \\
\triangle (I,2)  &=&  \triangle (I) A_2  &=&(v_1(I),v_2(I),v_3(I)) A_2.
\end{array}$$
Thus we have 

$$\begin{array}{rclrclrcl}
v_1(I,0)   &=& v_1(I), & v_2(I,0)   &=& v_2(I),  &
v_3(I,0)   &=& v_1(I)  + v_2(I)  + v_3(I) \\
v_1(I,1)   &=& v_2(I), &
v_2(I,1)   &=& v_3(I), &
v_3(I,1)   &=& v_1(I)  + v_2(I)  + v_3(I) \\
v_1(I,2)   &=& v_3(I), &
v_2(I,2)   &=& v_1(I), &
v_3(I,2)   &=& v_1(I)  + v_2(I)  + v_3(I).
\end{array}$$
Thus for any $I=(i_1,i_2, \ldots, i_n)$, we have 
\begin{eqnarray*}
\triangle(I)  &=& (v_1(I),v_2(I),v_3(I)) \\
&=& (1,1,1)A_{i_1}\cdots A_{i_n} \\
&=& (v_1,v_2,v_3))A_{i_1}\cdots A_{i_n} \\
&=& \triangle A_{i_1}\cdots A_{i_n} 
\end{eqnarray*}

For $I=(i_1, \ldots i_n)$ and $J=(j_1, \ldots , j_m)$, we say that $I<J$ if $n<m$ or, when $n=m$, there is an integer $k$ such that 
$1\leq k \leq n$ and      $i_l=j_l$, for all integers $1\leq l <k$, and $i_k<j_k.$  Thus we have the ordering
$$(0)<(1)<(2)<(00)<(01)<(02)< (10)<(11)<\cdots <(000) <(001) < \cdots  $$and hence an ordering of triples
$$\triangle    , \triangle(0), \triangle(1), \triangle(2),\triangle(00), \triangle(01), \triangle(02),\triangle(10),\dots, \triangle(000),\triangle(001), \ldots  . $$
For each triple $\triangle(I)$, we have the ordered list   of three numbers $v_1(I),v_2(I), v_3(I).$  Stern's triatomic sequence is simply putting in these three numbers for each triple $\triangle(I)$:

\begin{definition} {\rm Stern's triatomic sequence} is the sequence formed from all of the triples $\triangle(I)$ using the above ordering on the indices $I$ and is hence 
$v_1,v_2, v_3$,  $ v_1(0),v_2(0), v_3(0)   $,    $v_1(1),v_2(1)  $,    $v_3(1),v_1(2),v_2(2), v_3(2)   $,   $v_1(00),v_2(01), v_3(02), \ldots $.

\end{definition}
Thus the first terms are 
$    1, 1,1$,$
1,1,3  $, $  1,1,3  $,$ 1,1,3,
    1,1,5$,  $ 1, 3, 5  $, $  3, 1, 5           $,  $             1,1,5  $, $ 1, 3, 5    $, $ 3, 1, 5        $,  $    1,1,5  $, $1, 3, 5 $,  $    3, 1, 5.$

    \begin{definition}  If $I= (i_1, \ldots , i_n)$, then the three numbers $v_1(I),v_2(I), v_3(I)$ are said to be in \rm{level $n$}.
    \end{definition}
    
    Thus the level $0$ elements are $v_1,v_2,v_3$, or simply $1,1,1.$
   The level $1$ elements of the triatomic sequence  are $v_1(0),v_2(0), v_3(0), v_1(1),v_2(1), v_3(1),v_1(2),v_2(2), v_3(2)$ and hence the nine numbers
  $$1,1,3,1,1,3,1,1,3,$$
    while the level $2$ elements come from  the nine triples 
    $$\triangle(00),\triangle(01),\triangle(02),\triangle(10),\triangle(11),\triangle(12),\triangle(20),\triangle(21),\triangle(22)$$
   and hence are  the $27$ numbers    $$1,1,5,  1, 3, 5,    3, 1, 5,                  1,1,5, 1, 3, 5,  3, 1, 5,         1,1,5,  1, 3, 5,    3, 1, 5.$$
   The naturalness of ``levels''  will become more apparent in the next section when we describe the triatomic sequence in terms of a partitioning of a triangle.   Further, in section \ref{sec:Combinatorial Properties of Tri-Atomic Sequences}, we will see that many of the combinatorial properties of the triatomic sequence deals with which numbers are in different levels.
    
    We can write out  Stern's triatomic sequence via a recursive formula, which will take a bit to define.
Start with three numbers $a_1,a_2,a_3,$  which for us are the three numbers $1,1,1$.
For each positive integer $N \geq 4$,  there is a unique tuple $I=(i_1,i_2, \ldots , i_n)$, with each $i_j$ a zero, one or two, and a $k$ chosen to be as  large as possible from $k\in \{1,2,3\} $ such that 
\begin{eqnarray*}
N&=& 3(1 + 3 + 3^2 + \ldots + 3^{n-1}) + i_1 3^n + i_{2}3^{K-1} + \cdots + i_n 3 + k\\
&=& \frac{3(3^n-1)}{2} + i_1 3^n + i_{2}3^{n-1} + \cdots + i_n 3 + k.
\end{eqnarray*}
Define the   function
$\tau$ from $\mbox{tuples}\times \{1,2,3\} $ to  positive integers greater than three   by setting
$$\tau (I:k) =\tau (j_n,\ldots, j_1;k) = 3(1 + 3 + 3^2 + \ldots + 3^{n-1}) + j_n 3^n + j_{n-1}3^{n-1} + \cdots + j_1 3 + k.$$
Then we have for example that 
$$\tau(0;2)= 3 + 0\cdot 3 + 2 = 5$$
while 
$$\tau(2,1;3) = 3(1+ 3) + 2\cdot 3^2 + 1\cdot 3 + 3 = 36.$$
We extend the map $\tau$ to when we have the ``empty'' tuple, namely we define
$$\tau(\emptyset; 1) = 1, \tau(\emptyset; 2) = 2, \tau(\emptyset; 3) = 3.$$

We  write our eventual sequence  $a_1,a_2, a_3, \ldots $  as
$$a_{\tau(\emptyset; 1)}, a_{\tau(\emptyset; 2)}, a_{\tau(\emptyset; 3)}$$
$$a_{\tau(0;1)},a_{\tau(0;2)}, a_{\tau(0;3)},  a_{(1;1)},  a_{\tau(1;2)},  a_{(\tau1;3)},  a_{\tau(2;1)}, a_{\tau(2;2)}, a_{\tau(2;3)}$$
$$\vdots$$

We define the following three functions 
$$\tau_1:\N \rightarrow \N, \tau_2:\N \rightarrow \N, \tau_3:\N \rightarrow \N$$
 by setting
$$\begin{array}{ccccc}
\tau_1(N)  &=& \tau_1(\tau(j_n, \ldots , j_1; k))
&=& \tau(j_n, \ldots, j_2; k)\\
\tau_2(N)  &=& \tau_2(\tau((j_n, \ldots , j_1; k))
&=&\tau ( j_n, \ldots, j_2,;2) \\
\tau_3(N)  &=& \tau_3(\tau( j_n, \ldots , j_1; k))
&=& \tau( j_n, \ldots, j_2;3),
\end{array}$$
for
$$N=\tau(j_n, \ldots , j_1; k).$$
\begin{definition} Given three numbers $a_1,a_2$ and $a_3$, the corresponding {\rm Stern triatomic sequence} is 
$$a_N =\begin{cases} a_{\tau_1(N) },  &\mbox{if} \quad  j_1=0, l=1;  \\ 
 a_{\tau_2   (N) },  &\mbox{if} \quad  j_1=0, l=2;   \\
  a_{\tau_1   (N) } + a_{\tau_2   (N) } + a_{\tau_3   (N) },  &\mbox{if} \quad  j_1=0, l=3 ; \\
  a_{\tau_2(N) },  &\mbox{if}\quad  j_1=1, l=1;  \\ 
 a_{\tau_3   (N) },  &\mbox{if} \quad  j_1=1, l=2 ;  \\
  a_{\tau_1   (N) } + a_{\tau_2   (N) } + a_{\tau_3   (N) } , &\mbox{if} \quad  j_1=1, l=3; \\
    a_{\tau_3(N) },  &\mbox{if} \quad  j_1=2, l=1 ; \\ 
 a_{\tau_1   (N) } , &\mbox{if} \quad  j_1=2, l=2  ; \\
  a_{\tau_1   (N) } + a_{\tau_2   (N) } + a_{\tau_3   (N) } , &\mbox{if} \quad  j_1=2, l=3 .
    \end{cases}$$
     \end{definition}

     Of course, we want to make sure that this definition agrees with our earlier one.  Luckily the following proposition is a straightforward calculation.
     \begin{prop} We have for $N=\tau(I;k)$ that 

   $$a_N = v_I(k).$$
   \end{prop}
   
  The proof is just an unraveling of the definitions.

We also have a description for  the series $\sum a_Nx^N.$
Start with  three column vectors  $v_1, v_2, v_3$  in $\R^3$   and put them into the matrix
$$V=   \left(  \begin{array}{ccc}  v_1 & v_2 & v_3 \end{array} \right)  .$$

Set 
$$P(x) = A_0 + xA_1 + x^2 A_2.$$
Let $e_i$ be the standard three by one column vector with entries all zero save in the ith entry, which is one.  
\begin{theorem}When 
$$v_1 = \left( \begin{array}{c} 0 \\ 0 \\ 1 \end{array}  \right) , v_2 = \left( \begin{array}{c} 1 \\ 0 \\ 1 \end{array}  \right) , v_3 = \left( \begin{array}{c} 1 \\ 1 \\ 1 \end{array}  \right), $$ then
$\sum a_Nx^N $ is
$$e_3^T  \cdot V\cdot \left[I + \sum_{k=1}^{\infty}x^{\frac{3( 3^{k}-1)}{2} }P\left( x^{3^k} \right) P\left( x^{3^{k-1}}\right)\cdots  P\left( x^{3}  \right)  \right] \left( xe_1+ x^2e_2+ x^3  e_3 \right).$$
\end{theorem}
(For this  theorem, we are only interested in the bottom row of the matrix  $V$, which is $\triangle = (1,1,1)$.  The matrix $V$, though, will later be important when we link Stern's triatomic sequence to the Farey multidimensional continued fraction.)

\begin{proof}
The proof is again an unraveling of the notation, as follows.

Multiplying any three-by-three matrix by $ e_3^T $ on the left picks out the bottom row of the matrix.  For any tuple $I=(i_1, \ldots, i_n)$, with each $i_k$  a $0$, $1$ or $2$, we have 
$$(v_1(I),v_2(I),v_3(I)) = (1,1,1) A_{i_1}\cdots A_{i_n} = e_3^T  \cdot V A_{i_1}\cdots A_{i_n}.$$
In particular, $$e_3^T  \cdot V = (v_1,v_2,v_3)= \triangle.$$
Then 
\begin{eqnarray*}
e_3^T  \cdot V \cdot I \cdot  \left( xe_1+ x^2e_2+ x^3  e_3 \right) &=&   v_1 x+ v_2 x^2 + v_3x^3 \\
&=&a_1x+a_2x^2+a_3x^3,
\end{eqnarray*}
which are the desired first three terms in the series.
Similarly,
$$e_3^T  \cdot V \cdot P(x^3) \cdot  \left( xe_1+ x^2e_2+ x^3  e_3 \right) $$  is 
$$e_3^T  \cdot V \cdot (A_0 + x^3A_1 + x^6 A_2) \cdot  \left( xe_1+ x^2e_2+ x^3  e_3 \right),$$
which in turn equals
$$ (\triangle(0) + x^3\triangle(1) + x^6\triangle(2))    \left( xe_1+ x^2e_2+ x^3  e_3 \right).$$ Thus
\begin{eqnarray*}
e_3^T  \cdot V \cdot P(x^3) \cdot  \left( xe_1+ x^2e_2+ x^3  e_3 \right) 
&=&v_1(0)x + v_2(0) x^2 + v_3(0)x^3 \\
&&+ v_1(1)x^4 + v_2(1) x^5 + v_3(1)x^6 \\
&& + v_1(2)x^7 + v_2(2) x^8 + v_3(2)x^9 \\
&=& a_4x + a_5 x^2 + a_6x^3 \\
&&+ a_7x^4 + a_8(1) x^5 + a_9(1)x^6 \\
&& + a_{10}(2)x^7 + a_{11}(2) x^8 + a_{12}(2)x^9 .
\end{eqnarray*}
After multiplying through by the factor of $x^3$, we get the next nine terms in the series.

For the next $27$ terms, we start with
\begin{eqnarray*}
P(x^9)P(x^3) &=& (A_0 + A_1x^9 + A_2x^{18})(A_0 + A_1x^3 + A_2x^{6}) \\
&=& A_0A_0 + A_0A_1 x^3+ A_0A_2x^6 \\
&&+ A_1A_0x^9 + A_1A_1x^{12}+A_1A_2x^{15} \\
&&+  A_2A_0x^{18} + A_2A_1x^{21}+A_2A_2x^{24}.\\
\end{eqnarray*}
Then
\begin{eqnarray*}
e_3^T  \cdot VP(x^9)P(x^3)  &=& \triangle(00) +  \triangle(01) x^3+  \triangle(02)x^6 \\
&&+  \triangle(10)x^9 +  \triangle(11)x^{12}+ \triangle(12)x^{15}\\
&& + \triangle(20)x^{18} +  \triangle(21)x^{21}+ \triangle(22)x^{24}
\end{eqnarray*}
and hence 
\begin{eqnarray*}
e_3^T  \cdot VP(x^9)P(x^3)  \left( xe_1+ x^2e_2+ x^3  e_3 \right)  &=&v_1(00) x + v_2(00) x^2 + v_3(00) x^3 \\
&&+ v_1(01) x^4 + v_2(01) x^5 + v_3(01) x^6   \\
&&+ v_1(02) x^7 + v_2(02) x^8 + v_3(02) x^9  \\
&&+ v_1(10) x^{10} + v_2(10) x^{11} + v_3(10) x^{12}  \\
&&+ v_1(11) x^{13} + v_2(11) x^{14} + v_3(11) x^{15}  \\
&&+ v_1(12) x^{16} + v_2(12) x^{17} + v_3(12) x^{18}  \\
&&+ v_1(20) x^{19} + v_2(20) x^{20} + v_3(20) x^{21}  \\
&&+ v_1(21) x^{22} + v_2(21) x^{23} + v_3(21) x^{24} \\
&& + v_1(22) x^{25} + v_2(22) x^{26} + v_3(22) x^{27} 
\end{eqnarray*}
which is 
$$a_{13}x + a_{14} x^2 + a_{15} x^3 
+\cdots +  
 a_{37} x^{25} + a_{38} x^{26} + a_{39}x^{27} $$
Using that for $k=2$ 
$$\frac{3( 3^{k}-1)}{2} = 12,$$
we multiply the above sum by the factor of $x^{12}$, which does indeed give us the next $27$ terms.

The next $81$ terms are similarly computed.  To see how this would work, note that $P(x^{27})P(x^9)P(x^3)$ equals 
$$(A_0 + A_1x^{27} + A_2x^{54})(A_0 + A_1x^9 + A_2x^{18})(A_0 + A_1x^3 + A_2x^6),$$
which gives us
\begin{eqnarray*}
P(x^{27})P(x^9)P(x^3) 
&=& A_0A_0A_0 + A_0A_0A_1 x^3+ A_0A_0A_2x^6 \\
&&+ A_0A_1A_0x^9 + A_0A_1A_1 x^{12}+ A_0A_2A_2x^{15}  \\
&&+ A_0A_2A_0x^{18} + A_0A_2A_1 x^{21}+ A_0A_2A_2x^{24}  \\
&&+ A_1A_0A_0x^{27} + A_1A_0A_1x^{30}+ A_1A_0A_2x^{33}  \\
&&+A_1A_1A_0x^{36} + A_1A_1A_1x^{39}+A_1A_1A_2x^{42}\\
&&+A_1A_{2} A_0  x^{45} + A_1A_2 A_1x^{48}+ A_1A_2A_2x^{51}  \\
&&  + A_2A_0A_0x^{54} +A_2A_0A_1x^{57} +A_2A_0A_2x^{60} \\
&& + A_2A_1A_0x^{63}  +A_2A_1A_1x^{66}  +A_2A_1  A_2  x^{69}\\
&& + A_2A_2 A_0x^{72}+ A_2A_2A_1x^{75} + + A_2A_2A_2x^{78}
\end{eqnarray*}
Then $e_3^T\cdot VP(x^{27})P(x^9)P(x^3)$ will be a sum of row vectors of the various $\triangle(i_1,i_2,i_3)$, indexed by appropriate powers of $x$.  Then 
 $$e_3^T\cdot VP(x^{27})P(x^9)P(x^3) \left( xe_1+ x^2e_2+ x^3  e_3 \right), $$
 multiplied by $x^{39}$ (since for $k=3$ we have 
 $3( 3^{k}-1)/2= 39),$
 will give us the next $81$ terms
 $$a_{40} x^{40}  + \cdots+a_{118}x^{118}.$$
 The rest of the series is similar.

\end{proof}

  \section{Stern's Triatomic Sequence and Triangles}\label{sec:Stern's Triatomic Sequence and Triangles}
  In this section we will show that   Stern's triatomic sequence can be interpreted as a subdivision of a triangle, in analog to how Stern's diatomic sequence can be interpreted as systematically subdividing the unit interval.

Start with a  triangle $\triangle$ 
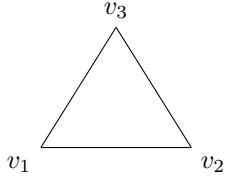
\begin{center}
  
\begin{tikzpicture}[scale=.4]

\draw(2.5,4)--(0,0);
\draw(2.5,4)--(5,0);

  \draw (0,0)--(5,0);

  \node[below left] at (0,0){$v_1$};
   \node[below right] at (5,0){$v_2$};
    \node[above] at (2.5,4){$v_3$};

 \end{tikzpicture}
 
\end{center}

\noindent with vertices $v_1,v_2,v_3$.  In this paper we will consider both the case when the vertices denote numbers and when the vertices denote vectors.  From $\triangle$ form three new triangles 

 
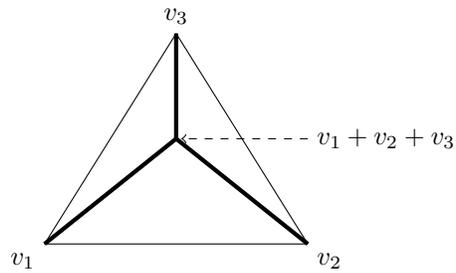
\begin{center}
  
\begin{tikzpicture}[scale=.7]
\draw[ ultra thick](2.5,2)--(2.5,4);
\draw(2.5,4)--(0,0);
\draw(2.5,4)--(5,0);

  \draw (0,0)--(5,0);
  \draw[ ultra thick](2.5,2)--(0,0);
  \draw[ ultra thick](2.5,2)--(5,0);
  
  \draw[->, dashed]( 5,2)--(2.6,2);
  
  \node[below left] at (0,0){$v_1$};
   \node[below right] at (5,0){$v_2$};
    \node[above] at (2.5,4){$v_3$};
    
    \node[right] at (5, 2){$v_1+v_2+v_3$};
 \end{tikzpicture}
 
\end{center}

\noindent where $\triangle(0)$ has vertices $v_1, v_2, v_1+v_2+v_3$,  $\triangle(1)$ has vertices $v_2, v_3, v_1+v_2+v_3$ and  $\triangle(2)$ has vertices $v_3, v_1, v_1+v_2+v_3.$
Each of these triangles can be divided into three more triangles, giving us nine triangles 
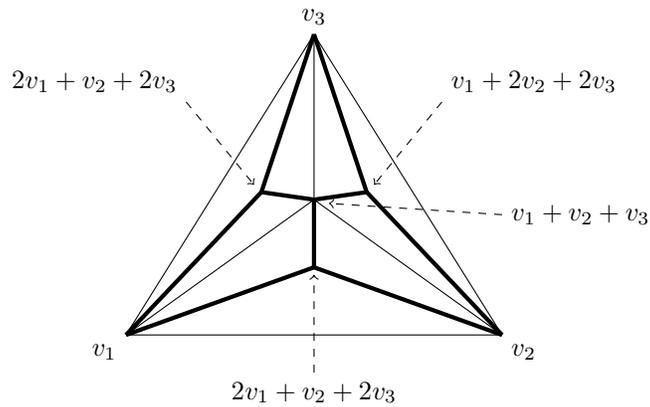
\begin{center}
  
\begin{tikzpicture}
\draw(2.5,4)--(0,0);
\draw(2.5,4)--(5,0);
  \draw (0,0)--(5,0);
  
  \draw(2.5,1.8)--(2.5,4);
 \draw(2.5,1.8)--(0,0);
  \draw(2.5,1.8)--(5,0);
  
  \draw[ ultra thick](2.5, .9)--(2.5,1.8);
   \draw[ ultra thick]   (2.5, .9)--(0,0);
    \draw[ ultra thick](2.5, .9)--(5,0);

    \draw[ ultra thick](1.8, 1.9)--(2.5, 1.8);
    \draw[ ultra thick](1.8, 1.9)--(0,0);
    \draw[ ultra thick](1.8, 1.9)--(2.5, 4);
    
    \draw[ ultra thick](3.2, 1.9)--(5,0);
     \draw[ ultra thick](3.2, 1.9)--(2.5,4);
      \draw[ ultra thick](3.2, 1.9)--(2.5,1.8);
      
       \draw[->, dashed]( 5,1.6)--(2.7,1.75);
       
        \draw[->, dashed]( 4.2,3.1)--(3.3, 2);
        
          \draw[->, dashed](.8, 3.1)--(1.7, 2);
           \draw[->, dashed](2.5, -.5)--(2.5, .8);
  
  \node[below left] at (0,0){$v_1$};
   \node[below right] at (5,0){$v_2$};
    \node[above] at (2.5,4){$v_3$};
    
    \node[right] at (5, 1.6){$v_1+v_2+v_3$};
    
    \node[above right] at (4.2, 3.1){$v_1+2v_2+2v_3$};
    
     \node[above left] at (.8, 3.1){$2v_1+v_2+2v_3$};
      \node[below] at (2.5, -.5){$2v_1+v_2+2v_3$};

 \end{tikzpicture}
 
\end{center}

We will identify each of these triangles with their vertices.  Thus we start with our initial triangle (which we will say is at level $0$):
$$\triangle =(v_1, v_2, v_3).$$
At level one we have
\begin{eqnarray*}
\triangle(0)   &   =  &  (v_1, v_2,   v_1+v_2+ v_3)  \\
\triangle(1)   &   =  &  (v_2, v_3,   v_1+v_2+ v_3)  \\
\triangle(2)   &   =  &  (v_3, v_1,   v_1+v_2+ v_3)
\end{eqnarray*}

 \begin{center}
  
\begin{tikzpicture}
\draw(2.5,1.8)--(2.5,4);
\draw(2.5,4)--(0,0);
\draw(2.5,4)--(5,0);
  \draw (0,0)--(5,0);
  \draw(2.5,1.8)--(0,0);
  \draw(2.5,1.8)--(5,0);

  \node at (2.5, .9){$\triangle (0)$};
  \node at (3.2,1 .9){$\triangle (1)$};
   \node at (1.8,1 .9){$\triangle (2)$};

  \node[below left] at (0,0){$v_1$};
   \node[below right] at (5,0){$v_2$};
    \node[above] at (2.5,4){$v_3$};
  
 \end{tikzpicture}
 
\end{center}

while for level two we have
\begin{eqnarray*}
\triangle(00)   &   =  &  (v_1, v_2,   2v_1+ 2 v_2+ v_3)  \\
\triangle(01)   &   =  &  (v_2, v_1 + v_2 + v_3,   2v_1+ 2 v_2+ v_3)  \\
\triangle(02)   &   =  &  (v_1+v_2+v_3, v_1,   2v_1+ 2 v_2+ v_3)  \\
\triangle(10)   &   =  &  (v_2, v_3,   v_1+ 2 v_2+ 2v_3) \\
\triangle(11)   &   =  &  (v_3, v_1+v_2+ v_3,   v_1+ 2 v_2+ 2v_3) \\
\triangle(12)   &   =  &  (v_1+v_2+ v_3, v_2,   v_1+ 2 v_2+ 2v_3)  \\
\triangle(20)   &   =  &  (v_3, v_1,   2v_1+  v_2+ 2v_3)  \\
\triangle(21)   &   =  &  (v_1, v_1+v_2+ v_3,   2v_1+  v_2+ 2v_3) \\
\triangle(21)   &   =  &  ( v_1+v_2+ v_3,  v_3,  2v_1+  v_2+ 2v_3).
\end{eqnarray*}

\begin{center}
  \begin{tikzpicture}[scale=1.4]

\draw(2.5,4)--(0,0);
\draw(2.5,4)--(5,0);
  \draw (0,0)--(5,0);
  
  \draw(2.5,1.8)--(2.5,4);
 \draw(2.5,1.8)--(0,0);
  \draw(2.5,1.8)--(5,0);
  
  \draw(2.5, .9)--(2.5,1.8);
   \draw(2.5, .9)--(0,0);
    \draw(2.5, .9)--(5,0);

    \draw(1.8, 1.9)--(2.5, 1.8);
    \draw(1.8, 1.9)--(0,0);
    \draw(1.8, 1.9)--(2.5, 4);
    
    \draw(3.2, 1.9)--(5,0);
     \draw(3.2, 1.9)--(2.5,4);
      \draw(3.2, 1.9)--(2.5,1.8);

       \node at (2.5,.4){$\triangle(00)$};
        \node at (3.1,1){$\triangle(01)$};
 \node at (1.9,1){$\triangle(02)$};

  \node[right] at (5,1.2){$\triangle(12)$};
    \draw[->, dashed]( 5,1.2)--(3.1,1.6);
    
  \node[left] at (0,1.2){$\triangle(21)$};
   \draw[->, dashed]( 0,1.2)--(1.9,1.6);

   \node[right] at (4.5,2.1){$\triangle(10)$};
    \draw[->, dashed]( 4.5,2.1)--(3.5, 2);

    \node[left] at (0.5,2.1){$\triangle(20)$};
    \draw[->, dashed]( 0.5,2.1)--(1.5, 2);
    
    \node[right] at (3.8, 3){$\triangle(11)$};
      \draw[->, dashed]( 3.8,3)--(2.8, 2.4);

      \node[left] at (1.2,3){$\triangle(22)$};
        \draw[->, dashed]( 1.2,3)--(2.2, 2.4);

  \node[below left] at (0,0){$v_1$};
   \node[below right] at (5,0){$v_2$};
    \node[above] at (2.5,4){$v_3$};
 \end{tikzpicture}
 
\end{center}

Continuing this process produces a sequence of triples 
$$\triangle, \triangle(0),\triangle(1), \triangle(2),\triangle(00),\triangle(01),\triangle(02), \triangle(10), \ldots, $$
which, in a natural way, will give us last section's Stern's triatomic sequence.

\section{The Farey Map}\label{sec:The Farey Map}
\subsection{Definition}

Stern's diatomic sequence is linked to continued fractions \cite{Northshield1}.  (This can also be seen in how the diatomic sequence can be interpreted via the Farey decomposition of the unit interval.)  There is a multidimensional continued fraction algorithm which generates in an analogous fashion   Stern's triatomic sequence.
Beaver and the author \cite{Beaver-Garrity1} gave a (seemingly) new multi-dimensional continued fraction algorithm, with  the goal of finding finding a generalization of the Minkowski ?(x) function.  This  algorithm, though,   seems particularly well-suited to generalize Stern's diatomic sequence.  (Panti \cite{Panti1}used the Monkmeyer algorithm to generalize the Minkowski ?(x) function.)  Also, it is standard in the study of multidimensional continued fractions to pass back and forth between  cones in $\R^3$  and  triangle in $\R^2$, via the map, when $x\neq 0$,
$(x,y,z) \rightarrow (y/x, z/x).$  In this subsection we define the Farey map as an iteration of a triangle, while in the next subsection we will think of the Farey map as an iteration of the corresponding cone.

Let $$\triangle = \{(x,y) \in \R^2: 0\leq y \leq x \leq 1\}.$$ 
Partition $\triangle $ into  the three subtriangles
\begin{eqnarray*}
\triangle_0   & = & \{(x,y) \in \triangle :  1-2y \geq x-y \geq y \}  \\
\triangle_1   & = & \{(x,y) \in \triangle :  2x-1 \geq y \geq 1-x \}  \\
\triangle_2   & = & \{(x,y) \in \triangle :  1-2x+2y \geq 1-x \geq x-y \}.
\end{eqnarray*}
The Farey map $T:\triangle \rightarrow \triangle$ is given by three  one-to-one onto maps 
$$T_i:\triangle_i \rightarrow \triangle$$ 
defined via
\begin{eqnarray*}
T_0(x,y)   & = &  \left( \frac{x-y}{1-2y}, \frac{y}{1-2y} \right) \\
T_1(x,y) &=& \left( \frac{y}{2x-1}, \frac{1-x}{2x-1} \right) \\
T_2(x,y) &=&  \left( \frac{1-x}{1-2x+2y}, \frac{x-y}{1-2x+2y} \right)
\end{eqnarray*}
These maps are the Farey analog of the Gauss map for continued fractions.  

\subsection{Farey Map on Vertices}

Now we will identify  triangles with their corresponding cones in $\R^3$.  Let $\triangle$ be a cone  with three vertices $v_1, v_2, v_3$, each written as a column vector in $\R^3$.  Subdivide $\triangle$ into three sub-cones

$$\begin{array}{ccc}
\triangle(0) & \mbox{with vertices}&  v_1, v_2,v_1+v_2+ v_3 \\
\triangle(1) & \mbox{with vertices}& v_2, v_3,v_1+v_2+ v_3  \\
 \triangle(3) & \mbox{with vertices} &v_3, v_1, v_1+v_2+ v_3
 \end{array}$$

We have three maps
$$A_i:\triangle \rightarrow \triangle_i,$$
which we describe via matrix multiplication on the right.  

Recall that  
$$A_0=  \left(  \begin{array}{ccc}    1  & 0 & 1\\ 0 & 1 & 1 \\ 0 & 0 & 1\end{array}  
\right),   A_1 = \left(  \begin{array}{ccc}    0  & 0 & 1\\ 1 & 0 & 1 \\ 0 & 1 & 1\end{array}  
\right), A_ 2 = \left(  \begin{array}{ccc}    0  & 1 & 1\\ 0 & 0 & 1 \\ 1 & 0 & 1\end{array}  
\right).$$
For any three vectors $v_1$ $v_2$ and $v_3$, we have
$$(v_1,v_2.v_3)A_0= (v_1, v_2, v_1 + v_2 + v_3)$$
$$(v_1,v_2.v_3)A_1= (v_2, v_3, v_1 + v_2 + v_3)$$
$$(v_1,v_2.v_3)A_2= (v_3, v_1, v_1 + v_2 + v_3).$$

We now specify our initial three vectors, by setting, as before,

$$V= (v_1,v_2,v_3)= \left(\begin{array}{ccc} 0 & 1 &1  \\ 0 & 0 & 1 \\ 1 & 1 & 1 \end{array} \right).$$
Then $VA_0, VA_1$ and $VA_2$ 
is simply the Farey map's analog of the Farey subdivision of the unit interval  and can be interpreted as the matrix approach for the inverse maps of the three Farey maps $T_i:\triangle_i\rightarrow \triangle$ of triangles in $\R^2$  from the previous subsection.

Then we have 
\begin{theorem} For an $n$-tuple $I=(i_1,\ldots, i_n)$ where each $i_j$ is a zero, one or two, we have that 
$$VA_{i_1}\ldots A_{i_n} = \left(  \begin{array}{ccc} * & * & * \\ * & * & * \\ v_1(I)&v_ 2(I) & v_3(I) \end{array}  \right).$$
\end{theorem}
Thus  Stern's triatomic sequence is  the analog to  the fact that Stern's diatomic sequence can be interpreted as  keeping track of the denominators for the Farey decomposition of the unit interval.

\section{Combinatorial Properties of Tri-Atomic Sequences}\label{sec:Combinatorial Properties of Tri-Atomic Sequences}

We now develop  some  analogs of  the combinatorial properties for Stern diatomic sequences for triatomic sequences.

\subsection{Three-fold symmetry}

Property 4 in Lehmer \cite{Lehmer1} is that there is a two fold symmetry for Stern's diatomic sequences on each ``level''.  We will see that for triatomic sequences, we have a three-fold symmetry at each level.

\begin{prop}  For any 
$N=\tau(0, j_{2}, \ldots , j_{n};k),$  where each   $j_i=0, 1$ or $2$ and $k$ is as large as possible for $k\in \{1,2,3\},$ we have
$$a_N = a_{N+ 3^k} = a_{N+ 2\cdot 3^N}.$$
\end{prop}
Thus 
$$a_{\tau(0, j_{2}, \ldots , j_{n};k)} = a_{\tau(1, j_{2}, \ldots , j_{n};k)} = a_{\tau(2, j_{2}, \ldots , j_{n},;k)} .$$

\begin{proof}
At the $0$th level, we have
$$a_1= 1, \;a_2=1, \; a_3=1.$$
The $1$th level is
$$a_4, \ldots , a_{12},$$
which by   definition is
$$a_1, a_2, a_1 + a_2 + a_3, a_2, a_3,  a_1 + a_2 + a_3, a_3, a_1,  a_1 + a_2 + a_3,$$
which in turn is simply
$$1,1,3,1,1,3,1,1,3,$$
satisfying the desired relation.

The rest follows from the basic recursion relation.

\end{proof}

\subsection{Tribonacci numbers}
Recall that the tribonacci numbers are the terms in the sequence $\beta_1, \beta_2, \ldots$ given by the recursion relation
$$\beta_{k+3} = \beta_k + \beta_{k+1} + \beta_{k+2}.$$
Such sequences are intimately intwined with triatomic sequences.

 \begin{theorem} For any $n$-tuple $J$, the following subsequence is a tribonacci   sequence:
 $$a_{\tau(J:1)}, a_{\tau(J:2)}, a_{\tau(J:3)}, a_{\tau(J, 1:3)}, a_{\tau(J, 1,1:3)}, a_{\tau(J, 1,1,1:3)}, \ldots$$
 \end{theorem}
 
 \begin{proof}
 Start with the triple
 $$a_{\tau(J:1)}, a_{\tau(J:2)}, a_{\tau(J:3)}.$$
 Then the triple
 $a_{\tau(J, 1:1)}, a_{\tau(J,1:2)}, a_{\tau(J,1:3)}$
 is
 $$a_{\tau(J:2)}, a_{\tau(J:3)},  a_{\tau(J:1)}+   a_{\tau(J:2)}+  a_{\tau(J:3)}.$$
 
Next, the triple  $a_{\tau(J, 1,1:1)}, a_{\tau(J,1,1:2)}, a_{\tau(J,1,1:3)}$ is $$ a_{\tau(J:3)},a_\tau{(J,1:3)}, a_{\tau(J:2)}+  a_{(J:3)} + a_{(J,1:3)}.$$
 This process continues, since by definition we have 
 $$a_{(J,1^s:3)} = a_{(J,1^{s-3}:3)} + a_{(J,1^{s-2}:3)}  + a_{(J,1^{s-1}:3)},$$
 where $1^s$ is notation for $s$ ones in a row,
 giving us our result.

 \end{proof}
 
 \begin{corr} The standard tribonacci   sequence $$1,1,1,3,5,9,17,31,57,105, 193,\ldots $$
 is $$a_1, a_2, a_3, a_{\tau(1:3)}, a_{\tau(1,1:3)}, \ldots , a_{\tau(1^s:3)}, \ldots $$
 
 \end{corr}

\subsection{Sums at Levels}

For Stern's diatomic sequence, Stern showed that the sum of all terms of level $n$ is $3^n+1$ (see property 2 in \cite{Lehmer1}).  A similar sum exists for  Stern's triatomic sequence.

\begin{prop}The sum of all terms in  Stern's triatomic sequence at the $n$th level is $5^n\cdot 3.$ Thus for fixed $n$
$$\sum a_{\tau(J:k)} = 5^n \cdot 3,$$
where the summation is over all $n$-tuples $J= (j_1, \ldots j_n)$, for  $j_i=0, 1$ or $2$, and over $k=1,2,3.$
\end{prop}

\begin{proof} The $0$th level is the three numbers $1,1,1$, and hence the sum is $3$.   Now suppose we have 
at the $n$ level the number 
$N=\tau(j_1, \ldots, j_n: 1).$
Then the triple of numbers 
$$a_N, a_{N+1}, a_{N+2}$$
is in the $n$th level.  What is important is that this triple generates the  following nine terms in the $(n+1)$st level:
$$a_N, a_{N+1},   a_N+  a_{N+1}  +  a_{N+2},$$
$$a_{N+1}, a_{N+2},   a_N+  a_{N+1}  +  a_{N+2}$$
and
$$a_{N+2}, a_{N},   a_N+  a_{N+1}  +  a_{N+2},$$
whose sum is 
$$5(a_N+  a_{N+1}  +  a_{N+2}),$$
giving us our result.

\end{proof}

\subsection{Values at Levels}

We can write each positive integer $N$ as $N=\tau(j_1, \ldots , j_n;k)$, where each $j_i\in \{0,1,2\}$ and each $k\in \{1,2,3\}$.  We have determined what the corresponding $a_N$ is at the $n$th level.  We set 
$$\delta_n(m)= \mbox{number of times  $m$ occurs in the  $n$th level}.$$
We further set
$$\delta_n^1(m)= \mbox{number of times  $m$ occurs in  the $n$th level when $k=1$},$$
$$\delta_n^2(m)= \mbox{number of times  $m$ occurs in  the $n$th level when $k=2$},$$
and
$$\delta_n^3(m)= \mbox{number of times  $m$ occurs in  the $n$th level when $k=3$}.$$
At level $1$, our sequence is 
$$1,1,3,1,1,3,1,1,3,$$
in which case
$$\delta_1(1) = 6, \delta_1(3) = 3 , \delta_1(n)=0\; \mbox{for} \;  n\neq1,3,$$
$$\delta_1^1(1) = 3, \delta_1^1(3) = 1 , \delta_1^1(n)=0\; \mbox{for} \;  n\neq1,3,$$
$$\delta_1^2(1) = 3, \delta_1^2(3) = 1 , \delta_1^2(n)=0\; \mbox{for} \;  n\neq1,3,$$
$$\delta_1^3(1) = 0, \delta_1^2(3) = 3 , \delta_1^2(n)=0\; \mbox{for} \;  n\neq1,3.$$

We now discuss some formulas for these new functions.

First for a technical lemma:
\begin{lemma}Suppose we are at the $n$th level, for some 
$N=\tau(j_1, \ldots , j_n;1).$  Then the triple $$a_N,a_{N+1},a_{N+2}$$ will induce the following three triples on the $(n+1)$st level:
The triple corresponding to 
$$a_{\tau(j_1, \ldots , j_n,0;1)}, a_{\tau(j_1, \ldots , j_n,0;2)}, a_{\tau(j_1, \ldots , j_n,0;3)}$$ will be
$$    (a_N, a_{N+1}, a_N+a_{N+1}+a_{N+2}),  $$
the triple corresponding to 
$$a_{\tau(j_1, \ldots , j_n,1;1)}, a_{\tau(j_1, \ldots , j_n,1;2)}, a_{(j_1, \ldots , j_n,1;3)}$$ will be
$$  (a_{N+1}, a_{N+2}, a_N+a_{N+1}+a_{N+2}),  $$
and the triple corresponding to 
$$a_{\tau(j_1, \ldots , j_n,2;1)}, a_{\tau(j_1, \ldots , j_n,2;2)}, a_{\tau(j_1, \ldots , j_n,2;3)}$$ will be
$$ (a_{N+2}, a_{N}, a_N+a_{N+1}+a_{N+2})  .$$
\end{lemma}

We will say we have a triple $a_N,a_{N+1},a_{N+2}$ at level $n$ if 
$N=\tau(j_n, \ldots , j_1;1).$  
\begin{prop}  For all positive $n$ and $m$,
$$\delta_n(2m)=\delta_n^1(2m) = \delta_n^2(2m)=\delta_{n}^2(2m)=0.$$
\end{prop}
\begin{proof}
If 
$N=\tau(j_1, \ldots , j_n;1)$ at the $n$th level and if each term in the triple $a_N,a_{N+1},a_{N+2}$ is odd, then all of their  their descendants in the next level will still be odd, from the above lemma. 
Thus to prove the proposition, we just have to observe that at the initial level $0$, our numbers are 
$1,1,1,$ forcing all subsequent elements to be odd.
\end{proof}

\begin{prop} For any odd number $2n+1>1$, we have 
$$\delta_n^3(2n+1) = 9.$$
\end{prop} 

\begin{proof} We first show that if $N=\tau(0^n;1)$, then the corresponding triple is 
$(1,1,2n+1).$ The proof is by induction.  At the initial zero level, the initial triples $(1,1,1)$ as desired.  Then suppose  at the $n$th level the triple corresponding to $N=\tau(0^n;1)$ is the desired  $(1,1,2n+1).$  Then we have at the next level three triples descending from $(1,1,2n+1),$ namely $(1,1,2n+3), (1, 2n+1, 2n+3)$ and $(2n+1, 1, 2n + 3)$.  By the three-fold symmetry of the sequence, this means that 
$$\delta_{n+1}(2n+3) \geq 9.$$

Now to show that this number cannot be greater than $9$.  Suppose at level $n\geq 1$ there are only three triples of the form $(1,1,2n+1)$.  This is true at level one.  For any triple $(a_N, a_{N+1},a_{N+2})$ at level $n$ with $a_N$ or $a_{N+1}$ greater than one, then none of this triple's descendants will contribute to $\delta_{n+1}^3(2n+3).$  Then at the next level, we will only pick up three triples $(1,1,2n+3)$.   Putting all this together, means that the only triples at level $n$ that will have descendants contributing to $\delta_{n+1}^3(2n+3)$ must be of the form  $(1,1,2n+1)$, and each of these three triples  will contribute three $(2n+3)$s.  Thus we have our result.

\end{proof}

\begin{prop}
$$\delta_n(m) = \delta_{n+1}^1(m)   =   \delta_{n+1}^2(m) .$$
\end{prop}

\begin{proof}
This follows immediately from the construction of the sequence.
\end{proof}

\begin{prop}
$$\delta_{k+m}(2k+1) = 2^m \delta_k(2k+1).$$
\end{prop}

\begin{proof}
It is at the $k$th level where $2k+1$ can last occur in a $\delta^3$ term.  Hence this follows from the previous proposition, since we have 
\begin{eqnarray*}
\delta_{k+m}(2k+1) &=&\delta_{k+m}^1(2k+1) + \delta_{k+m}^2(2k+1)+ \delta_{k+m}^3(2k+1)   \\
&=& \delta_{k+m-1}(2k+1) +  \delta_{k+m-1}(2k+1) + 0 \\
&=& 2 \delta_{k+m-1}(2k+1).
\end{eqnarray*}
\end{proof}

\begin{prop}
$$\delta_n(2n+1) = 9 + 2 \delta_{n-1}(2n+1)  $$
\end{prop}

\begin{proof}
We have already seen that $\delta_n^3(2n+1)=9.$  Since
$$\delta_n(2n+1) = \delta_n^1(2n+1) + \delta_n^2(2n+1) + \delta_n^3(2n+1)=  \delta_n^1(2n+1) + \delta_n^2(2n+1) +9,$$
all we need show is that $\delta_n^1(2n+1) + \delta_n^2(2n+1)= 2 \delta_{n-1}(2n+1).$  But this follows from the fact that each triple $a,b,c$  at the $k-1$st level generates  for the $k$th level the  triples 
$a,b,a+b+a$, or $b,c,a+b+c,$ or $c,a,a+b+c$ and each number appearing in the $n-1$st level appears twice in the $k$th level.

\end{proof}

Of course, it is easy to generate a table of various $\delta_K(n)$.  The following is a list of some examples.

\begin{center}
Table $1$
$$\begin{array}{c||c|c|c|c}

(K,n) & \delta_K(n) & \delta_K^1(n) & \delta_K^2(n)  & \delta_K^3(n)\\
\hline
(0,1) & 3 & 1 & 1 & 1\\
\hline
(1,1) & 6 & 3 & 3 & 0  \\
\hline
(1,3) &3& 0&0&3\\
\hline
(2,1) & 12 & 6 & 6 & 0 \\
\hline
(2,3) & 6 & 3 & 3 & 0 \\
\hline
(2,5) & 9 & 0 & 0 & 9 \\
\hline(3,1) & 24 & 12 & 12 & 0 \\
\hline
(3,1) & 24&12&12&0\\
\hline
(3,3) & 12 & 6 & 6 & 0\\
\hline
(3,5) & 18 & 9 & 9 & 0\\
\hline
(3,7)& 9&0&0&9 \\
\hline
(3,9)&18&0&0&18\\
\hline
(4,1)&48&24&24&0\\
\hline
(4,3)&24&12&12&0\\
\hline
(4,5)&36&18&18&0  \\
\hline
(4,7)&18&9&9&0 \\
\hline
(4,9)& 45&18&18&9\\
\hline
(4,13)& 36&0&0&36 \\
\hline
(4,15)&18&0&0&18\\
\hline
(4,17)&18&0&0&18\\
\hline
(5,1)&96&48&48&0\\
\hline
(5,3)&48&24&24&0\\
\hline
(5,5)&72&36&36&0\\
\hline
(5,7) & 36&18&18&0\\
\hline
(5,9)&90&45&45&0\\
\hline
(5,11)&9&0&0&9\\
\hline
(5,13)& 72&36&36&0\\
\hline
(5,15) & 36&18&18 &0\\
\hline
(5,17) & 72&18&18&36 \\
\hline
(5,19)& 18&0&0&18 \\
\hline
(5,21) &18&0&0&18 \\
\hline
(5,23)& 18&0&0&18 \\
\hline
(5,25)& 72&0&0&72 \\
\hline
(5,29)&36&0&0&36\\
\hline
(5,31) & 18&0&0&18 

\end{array}$$
\end{center}

\section{Paths of a directed graph}\label{sec:Paths of a directed graph}
There is a simple interpretation of  Stern's triatomic sequence as the number of paths in a directed graph from three initial vertices.  Recall the diagrams in section 3.  Starting with a triangle $\triangle$ with  vertices $v_1,v_2,v_3$, we systematically constructed a subdivision of $\triangle$ so that at  level $n$, we have $3^n$ triangles.  It is quite easy to put the structure of a directed graph on this subdivision.  Suppose we have one of the subtriangles $\triangle(I)$, where recall $I$ is an $n$-tuple consisting of $0,1$ or $2$.  Its vertices are denoted by $v_1(I), v_2(I)$ and $v_3(I).$ Then for level $n+1$, we add one more vertex, denoted by $v_1(I) + v_2(I)+v_3(I)$ and three directed paths, one from each $v_1(I)$ to the new vertex.  Finally, note that there are no directed paths between any of the three initial vertices.  Recall our earlier notation of writing any positive integer $N$ as $N=\tau(j_1, j_{2}, \ldots j_n;k)$ if

$N=3(1 + 3 + 3^2 + + 3^{n-1} ) + j_{1}3^n + j_{2}3^{n -1} + \cdots  + j_n3 + k,$
   where each $j_i$ is zero, one or two and $l$ is chosen to be as large as possible from the set $\{1,2,3\}$.

\begin{theorem}
For any positive integer $N$,  the number $a_N$ is the number of paths from $v_1, v_2$ and $v_3$ to the $l$th vertex of the  $\triangle(j_1, \ldots , j_n),$ save for when the corresponding vertex of $\triangle(j_n, \ldots , j_1)$ is one of the initial $v_1, v_2$ or $v_3$.
\end{theorem}

The proof is straightforward.

\section{Questions}
The Farey map is only one type of multidimensional continued fraction.  (For background on multi-dimensional continued fractions, see Schweiger \cite{Schweiger1}).  Recently,  Dasaratha et al. \cite{DasarathaFlapanGarrityLeeMihailaNeumann-Chun-Peluse-Stoffregen1} produced a family of multi-dimensional continued fractions  which includes many previously well-known algorithms, though not the Farey map.  Using this family, Dasaratha et al. \cite{Dasaratha et al 12 d} study analogs of Stern's diatomic sequence.  Independently, Goldberg \cite{Goldberg12} has started finding analogs for the Monkemayer map.   It would also be interesting to find the Stern analogs in the language of Lagarias \cite{Lagarias93}.

Also, there are many properties of Stern's diatomic sequences whose analogs for the Farey map have yet to be discovered.  Probably the most interesting would be to find the analog of the link between Stern's sequence and the number of hyperbinary representations there are for positive integers, as discussed by Northshield \cite{Northshield1}.

Another interesting problem is to attempt a multidimensional analog of the recently found link between the Tower of Hanoi graph and Stern's  diatomic sequence by Hinz, Klav\v{z}ar, Milutinovi\'{c}, Parisse and  Petr \cite{AH}.

Finally, it would be interesting to extend the polynomial analogs of Stern's diatomic sequence (as in the work of Dilcher and Stokarsky \cite{Dilcher-Stokarsky07, Dilcher-Stokarsky09},  of Klav\v{z}ar, Milutinovi\'{c}
and Petr \cite{Klavzar-Milutinovic-Petr07}, of Ulas  \cite{Ulas-Ulas11, Ulas12}, of Vargas \cite{Vargas12}, and of Allouche and Mend\`{e}s France \cite{Allouche-Mendes France12}) to finding polynomial analogs for triatomic sequences.

\section{Acknowledgments}
We would like to thank the referee for major help in improving the exposition of this paper, and catching a number of errors.

\end{singlespacing}

\end{document}